\documentclass[12pt]{article}

\usepackage[backend=bibtex, doi=false, url=false, isbn=false]{biblatex}
\DeclareNameAlias{sortname}{last-first}

\usepackage{fullpage}
\usepackage[T1]{fontenc}
\usepackage{hyperref}
\usepackage{xcolor}
\definecolor{link-color}{rgb}{0.15,0.4,0.15}
\hypersetup{
  colorlinks,
  linkcolor = {link-color}, citecolor = {link-color},
}

\usepackage{graphicx}
\usepackage{caption}
\usepackage{subcaption}
\usepackage{hyperref}
\usepackage{amsmath,amsthm,amssymb}
\usepackage{bbm}
\usepackage{tabularx}
\usepackage{soul}

\newtheorem{theorem}{Theorem}[section]

\newtheorem{corollary}[theorem]{Corollary}
\newtheorem{lemma}[theorem]{Lemma}
\theoremstyle{definition}
\newtheorem{remark}[theorem]{Remark}

\newcommand{\R}{\mathbb{R}}
\newcommand{\N}{\mathbb{N}}
\renewcommand{\P}{\mathbb{P}}
\newcommand{\E}{\mathbb{E}}
\newcommand{\1}{\mathbbm{1}}

\newcommand{\e}{\mathbbm{e}}

\newcommand{\n}{\mathbf{\eta}}

\DeclareMathOperator{\re}{\mathbb{R}}

\DeclareMathOperator{\ee}{\boldsymbol{e}}

\bibliography{references.bib}

\title{Conditioning subordinators  embedded in Markov processes}
\author{Andreas E. Kyprianou\footnote{Department of Mathematical Sciences, University of Bath, Claverton Down, Bath, BA2 7AY, UK. Email: \texttt{a.kyprianou@bath.ac.uk}}, \
Victor Rivero\footnote{CIMAT A. C.,
Calle Jalisco s/n,
Col. Valenciana,
A. P. 402, C.P. 36000,
Guanajuato, Gto.,
Mexico. Email: \texttt{rivero@cimat.mx}},
\ Bat\i{} \c{S}eng\"ul\footnote{Department of Mathematical Sciences, University of Bath, Claverton Down, Bath, BA2 7AY, UK. Email: \texttt{b.sengul@bath.ac.uk} $^{ a}$Corresponding author.} $^{,a}$
}
\begin{document}
  \maketitle

  \begin{abstract}
  The running infimum of a L\'evy process relative to its point of issue is know to have the same range that of the negative of a certain subordinator. Conditioning a L\'evy process issued from a strictly positive value to stay positive may therefore be seen as implicitly conditioning its descending ladder heigh subordinator to remain in a strip. Motivated by this observation, we consider the general problem of conditioning a subordinator to remain in a strip. Thereafter we consider more general contexts in which subordinators embedded in the path decompositions of Markov processes are conditioned to remain in a strip.   
  \end{abstract}
  \section{Introduction}
  Let $\mathbb{D}$ denote the space of c\`adl\`ag functions $\omega:[0,\infty) \rightarrow \R \cup \{\Delta\}$ such that, defining $\zeta=\inf\{t \geq 0: \omega_t=\Delta\}$, we have $\omega(t)=\Delta$ for $t \geq \zeta$. We call $\Delta$ the cemetery state and think of $\omega$ as killed once it enters the cemetery state. The space $\mathbb D$ is equipped with the Skorokhod topology and for $t \geq 0$, we write $(\mathcal F_t:t \geq 0)$ for the natural filtration. The process $X=(X_t:t \geq 0)$ denotes the co-ordinate process on $\mathbb D$ and we let $(X,\P_x)$ denote the law of a non-constant L\'evy process started at $x \in \R$.

  In, what is by now considered, classical work, it was shown in \cite{C, CD} that, under mild assumptions, there exists a (super)harmonic function $h\geq 0$ such that, for $x>0$,
  \begin{equation}
    \left.\frac{{\rm d}\mathbb{P}^\uparrow_x}{{\rm d}\mathbb{P}_x}\right|_{\mathcal{F}_t} : = \frac{h(X_t)}{h(x)}\1_{\{t<\tau^-_0\}}, \qquad t\geq 0,
    \label{COMLP}
  \end{equation}
  characterises the law of a L\'evy process conditioned to stay non-negative, where  $\tau^-_0 = \inf\{t>0: X_t <0\}$. To be more precise, the resulting  (sub-)Markov   process, $(X, \mathbb{P}^\uparrow_x)$, $x>0$, also emerges through the limiting procedure,
  \[
    \mathbb{P}^\uparrow_x(A) : = \lim_{q\downarrow0}\mathbb{P}_x(A, t< \e_q \,|\,  \tau^-_0 > \e_q ), \qquad t\geq 0, A\in\mathcal{F}_t,
  \]
  where, for $q>0$, $\e_q : = q^{-1}\e$ such that $\e$  is an independent exponentially distributed random variable with unit mean. This result would normally be proved in the setting of diffusions using potential analysis. For the case of L\'evy processes the analogous theory was not readily available and so the work of \cite{C, CD} is important in that it shows how excursion theory can be used instead.

  In this paper, we are interested in exploring conditionings of subordinators, that is, L\'evy processes with non-decreasing paths. Moreover, we are also interested in similarities that occur when conditioning subordinators that are embedded in the path decomposition of other Markov processes. In this respect, it is natural to  understand how to condition a subordinator to remain below a given threshold. To see why, let us return to the setting of conditioning a L\'evy process to remain non-negative and explore the effect of the conditioning on the range of the process $\underline{X}_t : = \inf_{s\leq t}X_s$, $t\geq 0$.

  It is well understood that there exists a local time at 0 for the process $(X_t - \underline{X}_t : t \geq 0)$, which is Markovian; see for example Chapter VI of \cite{bertoin}. If we write this local time process by $(L_t : t\geq 0)$ and set $L^{-1}_t = \inf\{s>0 : L_s>t\}$, $t\geq 0$, then $H_t : = X_{L^{-1}_t}$, for $L^{-1}_t<\infty$ and $H_t : = -\infty$ otherwise, defines a killed stochastic process with cemetery state $\{-\infty\}$, known as the descending ladder height process, whose range $(-\infty,0]$ agrees with that of $(\underline{X}_t: t\geq 0)$. In particular, for $x>0$, the law of $H$ under $\mathbb{P}_x$ is such that $S_{t}:=x- H_t$, $t\geq 0$ is a (killed) subordinator issued from $x$. (In fact, the renewal function associated to this subordinator is precisely the function $h$ in \eqref{COMLP}.) Since, for each $t>0$, $L^{-1}_t$ is in fact a stopping time, one may  consider the conditioning  associated to $(X,\mathbb{P}_x^\uparrow)$, $x>0$,  when viewed through the stopping times $(L^{-1}_t : t\geq 0)$, to correspond to conditioning the subordinator $(S_t : t\geq 0)$, issued from $x$, to remain positive; or equivalently to conditioning $-H$ to remain in the interval $[0,x)$.

  With this  example of a conditioned subordinator in hand, we extract the problem into its natural general setting. In the next section, we show how conditioning a general subordinator to stay in a strip, say $[0,a]$ can be developed rigorously. Additionally we show that 
  this conditioning can be seen as the combined result of choosing a point in $a$ according to a distribution, which is built from the potential measure of the subordinator, and then further conditioning the subordinator to hit that point.  Moreover, in the setting of stable subordinators, appealing additionally to the theory of self-similarity, we can interpret the conditioning as the result of an Esscher change of measure in the context of the Lamperti transform. 
  
 In the spirit of observing the relationship between conditioning a L\'evy processes  to stay positive and the  conditioning of a key underlying subordinator in its path decomposition,   we look at the case of  conditioning a Markov process to avoid the origin beyond a fixed time.   A key element of the associated path decomposition will be role of conditioning inverse local time at the origin to remain in the interval $[0,a)$, with $a>0$ fixed. Finally in Section \ref{victor} we use the ideas from the previous sections to condition a L\'evy process, issued from the origin, to reach an overall maximum in $[0,b)$ in the time interval $[0,a)$. This is tantamount to conditioning its ascending ladder height and ascending ladder time, which is a bivariate subordinator, to stay in the time-space box $[0,a)\times[0,b)$.

The key mathematical principle that connects all three sections, as well as connecting with the historical theory of L\'evy processes conditioned to stay positive, is that each of the conditionings we consider pertains to a generalisation of the conditioning of an embedded subordinator to stay in a strip. Accordingly, features of the resulting conditioned process can be described via transformations in the spirit of a Doob $h$-transform that are reminiscent of the Doob $h$-transform that uses the subordinator potential, which corresponds to conditioning a subordinator to stay in a strip.

  \section{Conditioning a subordinator to stay in an interval}

  In the previous section, we outlined the standard notation for a L\'evy process $X$. Henceforth we shall assume that the process $X$ is a subordinator. That is to say, it has non-decreasing paths. We shall often be concerned with the setting that it is issued from the origin, in which case we write $\mathbb{P}$ in place of $\mathbb{P}_0$.  The law of $X$ is determined by a characteristic pair $(\kappa,\nu),$ with $\kappa\geq 0,$ and $\nu$ a measure on $(0,\infty)$ such that $\int_{x\in(0,\infty)}(1\wedge x)\nu(dx)<\infty.$ These are related to the law of $X$ via the Laplace exponent
  $$
  -\frac{1}{t}\log \E\left(\exp\{-\lambda X_{t}\}\right)=:\phi(\lambda)=\kappa\lambda+\int^{\infty}_{0}(1-e^{-\lambda x})\nu(dx),\qquad \lambda \geq 0, t>0.
  $$
  As usual, we will denote by $\overline{\nu}(\cdot)$ the tail L\'evy measure of $X$
  $$
  \overline{\nu}(x):=\nu(x,\infty),\qquad x>0.
  $$

  For $q \geq 0$ define the $q$-potential function of $X$ by
  \[
    U^{(q)}(x):= \int_0^\infty {\rm e}^{-q t} \P(X_t \leq x)\, {\rm d} t \qquad x \in \R.
  \]
  It is not too hard to show that for all $q \geq 0$ and $x \in \R$, $U^{(q)}(x)<\infty$. Note also that by monotone convergence we have  that $U^{(q)} \rightarrow U:=U^{(0)}$ uniformly on compacts as $q \downarrow 0$. The function $U$ is also known as the renewal function of the subordinator $X$.
  The next lemma follows trivially from the fact that both $t \mapsto X_t$ and $x \mapsto U(x)$ are increasing.

  \begin{lemma}\label{lem:martingale}
    For each $a >0$, the process
    \begin{equation}
      U(a-X_t)\1_{\{X_t <a\}}\qquad  t \geq 0,
      \label{supermg}
    \end{equation}
    is a supermartingale.
  \end{lemma}

  \subsection{Definition of the conditioned process}

  As a non-negative supermartingale, we may use \eqref{supermg} to develop a Doob $h$-transform. For the remainder of the section, we fix $a >0$. For $x \in [0,a]$, we define a new measure $\mathbb{P}^\downarrow_x$ on $\mathbb D$ as follows:
  \[
    \mathbb{P}^\downarrow_x(A; t < \zeta)=\E_x\left[ \frac{U(a-X_t)}{U(a-x)}\1_{\{X_t < a, \, A\}}\right], \qquad t\geq 0, A \in \mathcal F_t,
  \]
  which makes sense in view of Lemma \ref{lem:martingale} and the fact that $U(z)>0,$ for all $z>0$. We will often abbreviate this as
  \begin{equation}\label{eq:Pc_def}
    \frac{{\rm d}\mathbb{P}^\downarrow_x}{{\rm d}\P_x}\Big|_{\mathcal{F}_t} = \frac{U(a-X_t)}{U(a-x)}\1_{\{X_t < a\}}.
  \end{equation}
  Since \eqref{supermg} is a supermartingale, the process $(X,\mathbb{P}^\downarrow_x)$ is sub-Markovian. The main result below states that there is a sense in which we can think of $(X,\mathbb{P}^\downarrow_x)$ as the process $(X,\P_x)$ conditioned to remain below level $a$. Hereafter and unless otherwise stated, we will assume that the level $a>0$ is fixed.

  \begin{theorem}\label{thm:limiti_eq}
    For $q >0$ let $\e_q : = q^{-1}\e$, where $\e$ is an exponential random variable which is independent of $X$. Moreover,  let $\tau^+_a := \inf\{t \geq 0: X_t > a\}$. Then for any stopping time $T$ and any $A \in \mathcal F_T$,
    \[
      \mathbb{P}^\downarrow_x(A; T < \zeta)= \lim_{q \downarrow 0} \P_x(A; T< \e_q \,|\,\e_q < \tau^+_a).
    \]
  \end{theorem}
  \begin{proof}
    Let $T$ be a stopping time and fix $A \in \mathcal F_T$. Then
    \begin{align} \label{eq:condition_exp}
      \P_x(A ; T < \e_q \,|\,\e_q \leq \tau_a^+)& = \frac{\P_x(A;T<\e_q; X_{\e_q} \leq a)}{\P_x(X_{\e_q}\leq a)} \\ \nonumber
      &= \frac{\P_x(A ; T< \e_q; X_{\e_q} -X_{T} \leq a-X_T)}{\P_x(X_{e_q} \leq a)}\\ \nonumber
      &= \E_x\left[\1_{\{A, \, X_{T}\leq a,\, T<\e_q\}}\frac{\P_{0}(X'_{\e_q-T} \leq a-X_T\,|\,\{T<\e_q\}\cap\mathcal{F}_T)}{\P_{x}(X_{\e_q} \leq a)}\right] \\ \nonumber
      &=\E_x\left[\1_{\{A,\, X_{T}\leq a,\, T<\e_q\}}\frac{\P_{0}(X'_{\e^{\prime}_q} \leq a-X_T)}{\P_x(X_{\e_q} \leq a)}\right] 
    \end{align}
    where $X'$ is an independent copy of $X$ and $\e^{\prime}_q$ is a copy of $\e_q$ independent of $X^{\prime}$. In the first equality we have used the fact that $X$ is increasing, in the third equality we use the stationary independent increments and in the final equality we have used the lack of memory property of the exponential distribution. 

    Now we have that, for each $y\geq x$
    \[
      \P_x(X_{\e_q} \leq y) = \int_0^\infty q {\rm e}^{-qs}\P(X_s \leq y-x)\, ds = qU^{(q)}(y-x).
    \]
    Using the above and the fact that $U^{(q)}\rightarrow U$ uniformly on compacts as $q \downarrow 0$ we get  from \eqref{eq:condition_exp}  that
    \[
      \lim_{q \downarrow 0}\P_x(A; T < \e_q \,\,|\, \e_q \leq \tau_a^+)=\E_x\left[\frac{U(a-X_T)}{U(a-x)}\1_{\{A ,\,  X_{T\leq a,}\,T<\infty\}}\right].
    \]
    It follows that 
    \[
      \lim_{q \downarrow 0}\P_x(A; T < \e_q \,|\, \e_q \leq \tau_a^+)=\E_x\left[\frac{U(a-X_T)}{U(a-x)}\1_{\{X_T \leq  a,  T<\infty\} \cap A}\right]=\mathbb{P}^\downarrow_x(A; T<\zeta)
    \]
    as required.
  \end{proof}
  \subsection{Path decomposition of the conditioned subordinator}

  Let us momentarily refer back to the motivation for the conditioning in the Theorem \ref{thm:limiti_eq} that comes from  the setting of the descending ladder height process of a L\'evy process conditioned to stay positive.

  The so-called Williams path decomposition, see e.g. \cite{C}, states that the conditioned L\'evy process reaches a global minimum, whose law can be characterised by the renewal function of the descending ladder height  subordinator. Moreover, given the space-time point of the global minimum, the evolution of the path of the conditioned L\'evy process thereafter is equal in law to an independent copy of the conditioned process issued from the origin, but glued on to the aforesaid space-time point.

  For example, in the special case that $\mathbb{P}^\uparrow_x$ corresponds to  a Brownian motion conditioned to stay positive, the original setting where D. Williams observed this path decomposition,  $x-H_t$, $t\geq 0$, is nothing more than a unit drift. The global minimum is achieved once the Brownian motion, and hence the process $x-H_t$, $t\geq 0$, hits a uniformly chosen point in $[0,x]$. Thereafter, it behaves like a Bessel-3 process issued from $0$, which happens to correspond to the weak limit on Skorokhod space $\lim_{x\downarrow0}\mathbb{P}^\uparrow_x$, i.e. the law of Brownian motion conditioned to stay non-negative when issued from the origin.

  If we strip away the Brownian motion in the above description and focus only on its descending ladder process, we are left with the conditioning of a very simple subordinator, i.e. a pure linear drift, conditioned to stay in the interval $[0,x]$. Moreover, this is done by uniformly choosing a point in $[0,x]$ and killing the subordinator once it is absorbed it reaches this state.

  One sees the same phenomena
  for the case of conditioning a Poisson process  to stay in an interval. Let $N=(N_t:t \geq 0)$ be a rate $1$ Poisson process. Then it is not hard to show that $U(x)=\lfloor x\rfloor +1 $ for every $x \geq 0$. Thus using \eqref{eq:Pc_def}, for each $a\in\{2,3,\dots\}$, $x \in \{0,\dots, a\}$ and $n \in \{0,\dots,a\}$
  \[
    \P^\downarrow_x(N_t = n) = \E_x\left[\frac{a+1-N_t}{a+1-x}\1_{\{N_t = n\}}\right]= \P_x(N_t = n) \left(1-\frac{n}{a+1-x}\right).
  \]
  We see that we can describe the law of $N$ under $\P^\downarrow_0$ as follows. Let $u\in\{0,\dots,a\}$ be chosen uniformly at random. Then under $\P^\downarrow_0$, $N$ is a rate $1$ Poisson process killed when it first hits level $u$.

  In greater generally, when $X$ is an arbitrary subordinator, $(X,\mathbb{P}^\downarrow_x)$ is an increasing killed Markov process,  and we should expect to see a `terminal value', $X_{\zeta-}$. In the case of the previous two examples, the law of this terminal value is uniformly distributed. In greater generality, again guided by the Williams path decomposition for a general L\'evy process in \cite{C}, one would expect the terminal value $X_{\zeta-}$ to be $U$-uniformly distributed. We can  ask for the law of $(X,\mathbb{P}^\downarrow_x)$ conditionally on the value of this maximum. Given the examples above, one would expect that under $\mathbb{P}^\downarrow_x|\{ X_{\zeta -}= y\}$, for $y\in [0,a)$, when $X$ has infinite jump activity, it is conditioned to approach $y$ continuously.

  %
  %
  %
  %
  Our objective in this section is thus to describe the path decomposition of the process $(X,\mathbb{P}^\downarrow_x)$ in this spirit. We begin by finding the law of its terminal value.

  \begin{lemma}\label{lem:maximum}
    For $a >0$, we have the identity
    \[
      \mathbb{P}^\downarrow_x\left( X_{\zeta-} \leq y\right) = \frac{U(y-x)}{U(a-x)},\qquad  x \leq y\leq a .
    \]
  \end{lemma}
  \begin{proof}
    Fix $x \in [0,a]$ and $y \in [x,a]$. Recall that $\tau^+_y:= \inf\{t \geq 0: X_t > y\}$. Then from Theorem \ref{thm:limiti_eq} we have that
    \[
      \mathbb{P}^\downarrow_x(\tau^+_y < \zeta)=\lim_{q \downarrow 0} \P_x(\tau^+_y < \e_q\,|\, \e_q < \tau^+_a)=1-\lim_{q \downarrow 0} \frac{\P_x(X_{\e_q} \leq y)}{\P_x(X_{\e_q} \leq a)} = 1-\frac{U(y-x)}{U(a-x)}.
    \]
    The lemma now follows since $\mathbb{P}^\downarrow_x\left(X_{\zeta-} \leq y\right) = \mathbb{P}^\downarrow_x(\tau^+_y>\zeta)$.
  \end{proof}

  Now we describe the law of $(X,\mathbb{P}^\downarrow_x)$ conditionally on $X_{\zeta-}$. In order to do so we make the following assumption:
  \[
    \tag{DA}U({\rm d}x) \text{ has a continuous density } u(x) \text{ with respect to the Lebesgue measure on } (0,\infty). \label{ass1}
  \]
  Using Proposition 12 in Chapter I in \cite{bertoin} we get that there exists a version $\tilde u$ of the potential density $u$ such that the function $x \mapsto \tilde u(a-x)$ is excessive for $X$. Next we show that the continuity assumption ensures that $x \mapsto u(a-x)$ is also  excessive for $X$.

  \begin{lemma}\label{mgproperties}
    Assume (DA). The process $(u(a-X_t)\1_{\{X_t <a\}}:t \geq 0)$ is a $\P$-supermartingale. 
  \end{lemma}
  \begin{proof}
    Suppose that $f : \R \rightarrow \R$ is a positive and bounded measurable function. Then we have the following equalities for any $t \geq 0$,
    \begin{align*}
    \int_{(0,\infty)} {\rm d}y f(y) \E[u(y-X_t)\1_{\{X_t<y\}}] &= \E \left[\int_{(0,\infty)} {\rm d}y f(y) u(y-X_t)\1_{\{X_t<y\}} \right]\\
    &= \E \left[\int_{(0,\infty)} {\rm d}y f(y+X_t) u(y) \right]\\
    &=\int_{(0,\infty)} {\rm d}y \E_y[f(X_t)] u(y)\\
    &= \int_0^\infty {\rm d}s \int_{(0,\infty)} \P(X_s \in {\rm d}y) \E_y[f(X_t)]\\
    &= \int_0^\infty {\rm d}s \E[f(X_{t+s})\1_{\{X_s>0\}}];
    \end{align*}
    where in the second equality we have used the substitution $y'=y-X_t,$ then in fourth and fifth we have applied the definition of the potential measure and the Markov property at time $t,$ respectively.  Furthermore, since $f$ is positive and $X$ is non-decreasing we infer that the right most term in the above identity is bounded by above as follows
    \begin{align*}
    \int_0^\infty {\rm d}s \E[f(X_{t+s})\1_{\{X_s>0\}}]   & \leq \int_0^\infty {\rm d}s \E[f(X_{t+s})\1_{\{X_{t+s}>0\}}]\\
   &= \int_t^\infty {\rm d}s \E[f(X_{s})\1_{\{X_{s}>0\}}]\\
    &\leq  \int_{(0,\infty)}{\rm d}y u(y) f(y).
    \end{align*}
 Since this holds for any $f$ positive and measurable it follows that for every $t \geq 0$
    \[
    \E[u(y-X_t)\1_{\{X_t<y\}}] \leq u(y) \qquad \text{ for almost every } y >0.
    \]
  Let us prove that the above holds for all $y>0.$ Take $y >0$ and let $\varepsilon>0$ be small. Then it follows that there exists a point $y_\varepsilon \in [y-\varepsilon,y]$ such that
    \begin{equation}\label{eq:almost_every}
      \E[u(y_\varepsilon-X_t)\1_{\{X_t<y_\varepsilon\}}] \leq u(y_\varepsilon).
    \end{equation}
    Letting $\varepsilon \downarrow 0$ we see that the right hand side of \eqref{eq:almost_every} converges to $u(y)$ by continuity of $u$. The left hand side of \eqref{eq:almost_every} converges to $\E[u(y-X_t)\1_{\{X_t<y\}}]$ by dominated convergence theorem and the continuity $u$. This finishes the proof.
  \end{proof}


  For each $y>0$ and $x\in [0,y)$ define a new measure $\P_x^{\circ,y}$ by setting
  \begin{equation}
    \frac{{\rm d}\P^{\circ,y}_x}{{\rm d}\P_x}\Big|_{\mathcal F_t} = \frac{u(y-X_t)}{u(y-x)}\1_{\{X_t \leq y\}}.
    \label{COMcirc}
  \end{equation}
  Again referring to work on conditioned L\'evy processes in \cite{C}, we can guess that the above change of measure corresponds to conditioning the subordinator $X$ to be continuously absorbed at the point $y$. More precisely, we have the following result in the spirit of Proposition 3 of \cite{C}, whose proof we also mimic.

  \begin{theorem} Assume (DA). Then for all $0 \leq x<b <y\leq a$,
    \[
      \P^{\circ,y}_x(A, t< \tau^+_b) = \lim_{\varepsilon\downarrow0}\P_x(A, t< \tau^+_b \,|\, X_{\tau^+_y-}\geq y-\varepsilon), \qquad t\geq 0, A\in\mathcal{F}_t.
    \]
  \end{theorem}
  \begin{proof}
    Fix $0 \leq x<b <y$ and suppose that $\varepsilon<y-x$. Applying the Markov property at time $t$, we have
    \begin{equation}\label{eq:ph_first_eq}
      \P_x(A, t< \tau^+_b \,|\, X_{\tau^+_y-}\geq y-\varepsilon)=\E_x\left[\1_{\{A,\, t< \tau^+_{b}<\infty\}}\frac{\P_{X_{t}}(X_{\tau^+_y-}\geq y-\varepsilon)}{\P_x(X_{\tau^+_y-}\geq y-\varepsilon)}\right]
    \end{equation}
    where we note that $\tau^+_b<\infty$ thanks to our conditioning.
    Appealing to Proposition III.2 and Theorem III.5 in \cite{bertoin},
    \begin{equation}\label{eq:overshoot}
      \P_x(X_{\tau^+_y-}\geq y-\varepsilon) = \P(X_{\tau^+_{(y-x)-}} \geq y-x-\varepsilon) = \kappa u(y-x)+\int_0^\varepsilon u(y-x-v)\overline{\nu}(v){\rm d}v.
    \end{equation}
    Using the continuity of $u$ and \eqref{eq:overshoot} we get that for any $x'\in [0,y)$
    \[
    \lim_{\varepsilon \downarrow 0} \frac{\P_{x'}(X_{\tau^+_y-}\geq y-\varepsilon)}{\P_x(X_{\tau^+_y-}\geq y-\varepsilon)} 
    = \frac{u(y-x')}{u(y-x)},
    \]
    where, if $\kappa>0$, then the limit is easy to see and, if $\kappa = 0$, we can appeal to L'H\^{o}pital's rule. 
Furthermore, because $u$ is assumed to be continuous on $(0,\infty)$ we have that for any $0\leq x^{\prime}\leq b<y$
$$\lim_{\varepsilon\downarrow 0}\frac{\P_{x'}(X_{\tau^+_y-}\geq y-\varepsilon)}{\kappa + \int_0^\varepsilon \overline{\nu}(v) {\rm d}v}\leq \sup_{z\in [\frac{1}{2}(y-b),y]}u(z)<\infty.$$
Hence by bounded convergence we have that
    \[
    \lim_{\varepsilon \downarrow 0}\E_x\left[\1_{\{A,\, t< \tau_{b}^+<\infty\}}\frac{\P_{X_{t}}(X_{\tau^+_y-}\geq y-\varepsilon)}{\P_x(X_{\tau^+_y-}\geq y-\varepsilon)}\right] = \E_x\left[\1_{\{A,\, t< \tau_{b}^+<\infty\}} \frac{u(y-X_t)}{u(y-x)}\right].
    \]
    Taking limits in \eqref{eq:ph_first_eq} and comparing to above finishes the proof.
  \end{proof}

  The following theorem shows that, under the assumption (DA),
  conditioning a subordinator to stay in a strip may be seen as first picking a point $U$-uniformly in $[0,a)$, after which the subordinator is conditioned to continuously hit that point.

  \begin{theorem}\label{pathdecomp}
  Assume (DA) and let $0 \leq x<y<a$. Then conditionally on $X_{\zeta-}=y$ the law of $X$ under $\mathbb{P}^\downarrow_x$ is that of $X$ under $\P^{\circ,y}_x$.
  \end{theorem}
  \begin{proof}
    Fix $0 \leq x<y<a$.
  We start by observing that by Lemma~\ref{lem:maximum}, 
  $$\P^\downarrow_{x}(X_{\zeta-}\in {\rm d}y)=\frac{u(y-x)}{U(a-x)}\1_{\{x<y<a\}}{\rm d}y.$$ This fact, together with the Markov property at time $t$ under the measure $\P^{\downarrow},$ implies that for arbitrary $0\leq x<a,$ $t\geq 0,$ $A\in \mathcal{F}_{t},$ and $f:[0,\infty)\to [0,\infty)$ positive and measurable, we have that
\begin{equation*}
\begin{split}
\E^{\downarrow}_{x}\left(1_{\{A,\ t<\zeta\}}f(X_{\zeta-})\right)&=\E^{\downarrow}_{x}\left(1_{\{A,\ t<\zeta\}}\E^{\downarrow}_{X_{t}}\left(f(X_{\zeta-})\right)\right)\\
&=\int^{a}_{0}{\rm d}y f(y) \E^{\downarrow}_{x}\left(1_{\{A, t<\zeta\}}1_{\{X_{t}<y\}}\frac{u(y-X_{t})}{U(a-X_{t})}\right)\\
&=\int^{a}_{0}{\rm d}y f(y) \frac{u(y-x)}{U(a-x)}1_{\{x<y\}} \E_{x}\left(1_{\{A, t<\zeta\}}1_{\{X_{t}<y\}}\frac{u(y-X_{t})}{u(y-x)}\right)\\
&=\int^{a}_{0}{\rm d}y f(y) \frac{u(y-x)}{U(a-x)}1_{\{x<y\}}{\P^{\circ,y}_x}\left(A, t<\zeta \right);
\end{split}
\end{equation*}
where in the third equality we used the definition of the measure $\P^{\downarrow}$.
    Combing this with the law of $X_{\zeta-}$ we see that
    \[
      \E^\downarrow_x\left[f(X_{\zeta-})\P(A;t<\zeta |X_{\zeta-})\right]=\E\left[f(X_{\zeta-})\P^{\circ,X_{\zeta-}}_x\left(A, t<\zeta \right)\right]
    \]
  which concludes the proof.

  \end{proof}

  \subsection{Interpreting the self-similar case}\label{pssmp}

  To get another perspective on the pathwise behaviour of conditioned subordinators, let us restrict our attention to $\alpha$-stable subordinators, where the additional benefits of self-similarity can be explored.
  Recall that a subordinator $X$ is called $\alpha$-stable if for all $t \geq 0$ and $c >0$
  \begin{equation}\label{eq:alpha_scaling}
    (cX_{c^{-\alpha}t}: t\geq 0) \overset{d}{=}  (X_{ t}: t\geq 0),
  \end{equation}
  where it must necessarily hold that  $\alpha \in (0,1)$.
  Henceforth suppose that $(X,\P)$ is an $\alpha$-stable subordinator. In particular, we restrict ourselves to issuing the process from the origin without loss of generality in the forthcoming analysis. It is known that
  \[
    \E[{\rm e}^{-\lambda X_1}] ={\rm e}^{-C\lambda^\alpha},\qquad \lambda\geq 0,
  \]
  for some constant $C>0$. Without loss of generality, we henceforth assume that $C=1$.
  From \eqref{eq:alpha_scaling} it follows that $U(x)=x^\alpha U(1)$ for all $x \geq 0$ and  hence
  \begin{align}
    &\frac{{\rm d}\mathbb{P}^\downarrow}{{\rm d}\P}\Big|_{\mathcal F_t}=\left(\frac{a-X_t}{a}\right)^\alpha \1_{\{X_t \leq a\}},\label{eq:ss_condition}\\
    &\frac{{\rm d}\P^{\circ}}{{\rm d}\P}\Big|_{\mathcal F_t}=\left(\frac{a-X_t}{a}\right)^{\alpha-1} \1_{\{X_t \leq a\}}.\label{eq:ss_hit}
  \end{align}
  Our goal here is to give a different pathwise interpretation of $\mathbb{P}^\downarrow$ and $\P^{\circ,y}$ by considering the above changes of measure in the context of the Lamperti transform, see~ \cite{L}.

  For each $a>0$ and $t \geq 0$ define
  \[
    Y^{(a)}_t=
    \begin{cases}
      a-X_t & \text{ if } X_t < a,\\
      0& \text{ otherwise.}
    \end{cases}
  \]
  It is not hard to check that under each of the measures $\P$, $\mathbb{P}^\downarrow$ and $\P^{\circ,y}$, $Y^{(a)}$ is a positive-valued Markov process issued from $a$ with the following additional property: for every constant $c>0$,
  \[
    (cY^{(a)}_{c^{-\alpha} t}: t\geq 0)\overset{d}{=}(Y^{(ca)}_t:t \geq 0).
  \]
  Such Markov processes are known in the literature as positive self-similar Markov processes (pssMp). The classical Lamperti transform, \cite{L},  allows us to write
  \begin{equation}
    Y^{(a)}_t = a {\rm e}^{\xi_{\varphi(a^{-\alpha}t)}}, \qquad t<\varsigma: =  \inf\{s>0 : Y^{(a)}_s =0\},
    \label{Lampertirep}
  \end{equation}
  where $\xi=(\xi_t:t \geq 0)$ is the negative of a subordinator which is killed at an independent and exponentially distributed random time and
  \begin{equation}\label{eq:lamperti_time_change}
    \varphi(s):=\inf\left\{t >0: \int_0^t {\rm e}^{\alpha \xi_u}\,{\rm d}u > s\right\}.
  \end{equation}

  We describe how the three processes $(\xi,\P)$, $(\xi,\mathbb{P}^\downarrow)$ and $(\xi,\P^{\circ,y})$ are related. We first begin by identifying the Laplace exponent of the process $\xi$. The next result is known, as the process $Y$ is a special example of a stable process killed on exiting the lower half-line, which has been discussed e.g. in \cite{CC, HR}, however, we re-establish it here in a different way for convenience.

  \begin{lemma}\label{lem:laplace_killed}
    For $\lambda \geq 0$,
    \[
      \Phi(\lambda):=-\log\E[{\rm e}^{\lambda \xi_1}]=
      \frac{\Gamma(1+\lambda)}{\Gamma(1+\lambda-\alpha)}
    \]
  \end{lemma}
  \begin{proof}
    Let $\xi^*$ be the L\'evy process which  is equal in law to $\xi$ but without killing.  For $\lambda \geq 0$, let $\Phi^*(\lambda):=-\log\E[{\rm e}^{\lambda \xi^*_1}]$ and write $q$ for  the rate at which $\xi$ is killed. Note that
    \[
      \tau_a^+:=\inf\{t \geq 0: X_t >a\}= \inf\{t \geq 0: Y^{(a)}_t=0\}.
    \]
    Then it follows that for each $\lambda\geq 0$,
    \begin{equation}
      \frac{q}{q+\Phi^*(\lambda)}=\E[{\rm e}^{\lambda\xi^*_{\e_q}}]=\E[{\rm e}^{\lambda\xi_{\varsigma-}}] = \E\left[\left(\frac{a-X_{\tau^+_a-}}{a}\right)^{\lambda}\right].
      \label{develop}
    \end{equation}
    The random variable inside the expectation on the right hand side is known as an undershoot and it's law is given by
    \[
      \P\left(\frac{a-X_{\tau^+_a-}}{a} \in {\rm d}y\right)= \frac{y^{-\alpha}(1-y)^{\alpha-1}}{\Gamma(1-\alpha)\Gamma(\alpha)} \, {\rm d}y,\qquad y\in (0,a]
    \]
    see for example \cite[Exercise 5.8]{kyprianou}. Developing the right hand side of \eqref{develop} we get
    \begin{align*}
      \frac{q}{q+\Phi^*(\lambda)}&=\frac{1}{\Gamma(1-\alpha)\Gamma(\alpha)}\int_0^1 y^{\lambda-\alpha}(1-y)^{\alpha-1}\,{\rm d}y\\
      &=  \frac{1}{\Gamma(1-\alpha)\Gamma(\alpha)} \frac{\Gamma(1+\lambda -\alpha)\Gamma(\alpha)}{\Gamma(1+\lambda)} \\
      & = \frac{\Gamma(1+\lambda -\alpha)}{\Gamma(1-\alpha)\Gamma(1+\lambda)}, \qquad \lambda\geq0.
    \end{align*}
    Since $\Phi(\lambda)=q+\Phi^*(\lambda)$, we have that
    \begin{equation}\label{eq:before_killing}
      \Phi(\lambda)=q\frac{\Gamma(1+\lambda)\Gamma(1-\alpha)}{\Gamma(1+\lambda-\alpha)}
    \end{equation}
    and hence it suffices to show that $q=1/\Gamma(1-\alpha)$.

    To this end, let $\nu$ be the L\`evy measure of $(X,\P)$, then it is known that (\cite[Excercise 5.8 (i)]{kyprianou}) for any $x >0$, $\nu(x,\infty)=x^{-\alpha}/\Gamma(1-\alpha)$. The Poissonian structure of the jumps of $X$ implies that for any $t \geq 0$, on $\{X_{t-}<a\}$, the rate at which $X$ exceeds $a$, and hence the rate at which $Y^{(a)}$ is killed, is
    \[
      \nu(a-X_{t-},\infty) {\rm d}t = \frac{(Y^{(a)}_{t-})^{-\alpha}}{\Gamma(1-\alpha)}{\rm d}t.
    \]
    On the other hand, referring to the Lamperti representation \eqref{Lampertirep},  noting in particular that
    \[
      \int_0^{\varphi(a^{-\alpha}t)} {\rm e}^{\alpha\xi_s}{\rm d}s
      = a^{-\alpha}t, \qquad t<\varsigma,
    \]
    the process $Y^{(a)}$ is killed at rate
    \[
      q{\rm d}\varphi(a^{-\alpha}t) = q a^{-\alpha}{\rm e}^{-\alpha\xi_{\varphi(a^{-\alpha} t)}} {\rm d}t= q(Y^{(a)}_t)^{-\alpha}{\rm d}t.
    \]
    Comparing these two rates, we see that $q = 1/\Gamma(1-\alpha)$, and the proof is completed.
  \end{proof}

  Now noting that $\varphi(a^{-\alpha }t)$ is a stopping time in the natural filtration of $\xi$,
  if we now revisit the change of measures  \eqref{eq:ss_condition} and \eqref{eq:ss_hit}, we see that they are equivalent to performing exponential changes of measure with respect to the law of $\xi$ with the exponential (super)martingales
  \[
    {\rm e}^{\alpha\xi_t}
    \text{ and }
    {\rm e}^{(\alpha -1)\xi_t}, \qquad t\geq0,
  \]
  respectively. Note that the first of these two is a strict supermartingale on account of the fact that $\Phi(\alpha)>0$. The second is a martingale thanks to the convenience that $\Phi(\alpha-1) = 0$.
  Moreover, under these exponential changes of measure, we find the new Laplace exponents of $\xi$ become
  \[
    \Phi^{\downarrow}(\lambda)=\frac{\Gamma(1+\lambda +\alpha)}{\Gamma(1+\lambda)}
    \text{ and }
    \Phi^{\circ}(\lambda)= \frac{\Gamma(\alpha+\lambda)}{\Gamma(\lambda)},
    \qquad \lambda \geq 0,
  \]
  respectively.

  As we might expect, given that $(X, \mathbb{P}^\downarrow)$ is a killed Markov process, the corresponding pssMp, $Y^{(a)}$, has Lamperti transform which reveals a killed underlying subordinator $-\xi$. That is to say, $\Phi^\downarrow(0)>0$. Similarly as we know that $(X, \mathbb{P}^{\circ,a})$ is continuously absorbed at level $a$, the pssMp process $Y^{(a)}$ is continuously absorbed at the origin and hence, not surprisingly,  $\Phi^\circ(0)=0$.


  \section{Last passage by time \texorpdfstring{$a$}{a} for a Markov process }

  In this section we consider the following problem. Let $X=(X_t:t \geq 0)$ be a Markov process and $a>0$, then what does the process $X$ conditioned  to not visit $0$ after time $a$ look like? The motivation for the problem and connection with the first half of the paper comes from the following. Suppose that $Y=(Y_t:t \geq 0)$ is a subordinator which is not a pure drift and for $x \geq 0$ define
  \[
    D_x:=Y_{\tau^+_x}-x
  \]
  where $\tau^+_x:=\inf\{t\geq 0: Y_t>x\}$. Since $Y$ has the strong Markov property, it follows that $D=(D_x:x \geq 0)$ is a Markov process with the property that the closure of its zero set coincides with that of the image of $Y.$ 
 Hence it follows that conditioning the process $Y$ to stay in the interval $[0,a]$, as in the previous sections, is equivalent to conditioning the Markov process $D$ to not hit $0$ after time $a$. In this section we would like to extend this notion to more general Markov processes. Although the results are stated for Markov processes living on $\R$, they can be easily adapted to more general Polish spaces. Before doing so we first introduce some definitions and recall some useful facts.

We will assume that $X$ is a nice Markov process on $\R$ in the sense of Chapter IV in \cite{bertoin}. Denote $T_0:= \inf \{t \geq 0: X_t=0\}$ and suppose that
 \[
   \P_x(T_0<\infty)>0,\qquad x\in \R.
 \]
Let $L=(L_t:t \geq 0)$ be the local time of $X$ at $0$. In particular, $L$ is the unique process which increases on the set $\{s:X_s=0\},$ and hence there exists a $\beta \geq 0$ such that
 \begin{equation}\label{eq:local_def}
   \beta L_t= \int_0^t \1_{\{X_s=0\}}\, {\rm d}s,\qquad t\geq 0.
 \end{equation}
 For more on the existence and construction of  the local time process, see \cite[Section IV]{bertoin}. 
 Next, for $q\geq 0$, define
 \begin{equation}\label{eq:V_def}
   V^{(q)}_{s,t}(x):= \E_x\left[\int_s^t e^{-qu} {\rm d}L_u\right].
 \end{equation}
When $q=0$ the super-script in $V^{(q)}$ will be omitted for notational convenience.

 Next let $\mathcal E^*$ denote the  excursion set, that is the set of c\`adl\`ag paths $\epsilon: [0,\zeta] \rightarrow \R$ such that $\epsilon(t)\neq 0$ if and only if $t \in (0,\zeta)$ for some $\zeta=\zeta(\epsilon)>0$. There exists a  $\sigma$-finite measure $\n$ on $\mathcal E^*$ which is induced by the process $X$, known as the excursion measure, it allows to describe the excursions of $X$ from $0$ as follows, see e.g. Section 4 in Chapter IV, \cite{bertoin} for further background. Consider the set $U=[0,\infty)\backslash\overline{\{t:X_t= 0\}}$. Since this is an open set, it can be written as a countable union of disjoint intervals $\{(\ell_i,r_i)\}_{i \geq 1}$. Next for any $i \in \N$,
  \[
    \epsilon_i(t):=\begin{cases}
    X_{\ell_i+t} &\text{ if } 0<t \leq r_i -\ell_{i}\\
    0 & \text{ if }t > r_i-\ell_{i}.
    \end{cases}
  \]
  Notice that the Stieltjes measure ${\rm d}L_t$ is well defined because the process $L$ is non-decreasing. A key result in excursion theory states that
  \[
  \sum_{i\geq 1} \delta_{(\ell_i,\epsilon_i)}
  \]
  is a Poisson point process on $[0,\infty)\times \mathcal E^*$ with intensity given by $\E[{\rm d}L_t] \otimes \n$, see for example \cite[IV Theorem 10]{bertoin}. One consequence of this is the compensation formula which states the following. For $t \geq 0$, let $\mathbb D[0,t]$ be the space of c\`adl\`ag paths $\omega:[0,t]\rightarrow \R$. Consider a function $F=(F_u:u \geq 0)$ such that
  $F_u:\mathbb D[0,u] \times\mathcal E^* \rightarrow \R$ for which $u \mapsto F_u(\cdot, \epsilon)$ is adapted with respect to the filtration $\mathcal F$ and is left-continuous, for every $\epsilon \in \mathcal E^*$. Then
  \begin{equation}\label{eq:master_formula}
    \E\left[\sum_{i=1}^\infty F_{\ell_i}((X_t:t \leq \ell_i),\epsilon_{i})\right] = \E\left[\int_0^\infty \n(F_u((\epsilon_t:t \leq u),\epsilon)) {\rm d}L_u \right].
  \end{equation}
 Furthermore, under $\n$ the process $\epsilon$ has the Markov property with the transition semigroup of $X$ killed at its first hitting time of $0.$

 For $x \in \R$ and $q>0$ define
 \begin{equation}\label{eq:h_q}
   h_q(x):= \frac{\P_x(T_0> \e_q)}{q\beta +\n(1-e^{-q\zeta})}
 \end{equation}
 where $\e_q$ is an independent exponential with parameter $q$. In order to state our main theorem we must first make two assumptions:
  \begin{itemize}
    \item[(A)] for each $x \in \R$, $h(x):=\lim_{q \downarrow 0}h_{q}(x)$ exists,
    \item[(B)] either there exists a measurable function $H$ such that $|h_{q}(x)|\leq H(x),$ $x\in \R$ and $\sup_{t\geq 0}\n(H(\epsilon_t), t<\zeta)<\infty$, or that the mapping $q\mapsto h_{q}(x)$ is monotone for all $x\in \R$.
  \end{itemize}

  Now we can state the main theorem of this section.

  \begin{theorem}\label{thm:markov_limit}
    Let $a>0,$ and for $t \geq 0$, let $g_t:= \sup\{s \leq t: X_s=0\}$. Assume that (A) and (B) hold. Then for each $x \in \R$ there exists a measure $\P^{\leftarrow a}_x$ such that for any stopping time $T$ and $A \in \mathcal F_T$,
    \[
      \P^{\leftarrow a}_x(A; T< \zeta) = \lim_{q \downarrow 0}\P_x(A; T< \e_q | g_{\e_q} < a).
    \]
  \end{theorem}

  The next theorem describes the path of the process $(X,\P^{\leftarrow a}_x)$.

  \begin{theorem}\label{thm:markov_limit_desc}
    Assume that (A) and (B) hold. For $a>0$ and $x\in \R$ the measure $\P^{\leftarrow a}_x$ admits the representation
\begin{equation*}
\begin{split}
&\E^{\leftarrow a}_{x}\left[F(X_{s}, s< g_{\infty})f(g_{\infty})G(X_{v+g_{\infty}}, v\leq u)\right]\\
&=\frac{1}{V_{0,a}(x)}\E_{x}\left[\int^{a}_{0}{\rm d}L_{t}F(X_{s}, s<t) f(t)\n\left(G(\epsilon_{s}, s\leq u)h(\epsilon_{u}), 0<u<\zeta \right)\right],
\end{split}
\end{equation*}
for any $F,G:\mathbb{D}\to \re$ bounded, measurable functionals, and $f:[0,\infty)\to[0,\infty)$ measurable.
\end{theorem}

The description in Theorem \ref{thm:markov_limit_desc} immediately allows us to decompose the path of $(X,\P^{\leftarrow a}_x)$ as follows.

  \begin{corollary}\label{cor:1}
    Under $\P^{\leftarrow a}_x$, $g_\infty=\sup\{t \geq 0: X_t=0\}<\infty$ almost surely, and
    \[
      \P^{\leftarrow a}_x(g_\infty > t)=\frac{V_{t,a}(x)}{V_{0,a}(x)} \qquad t \leq a.
    \]
The process $(\P^{\leftarrow a}, X)$ is obtained as the concatenation of two independent Markov processes $(X_t: t \leq g_\infty)$ and $(X_t: t \geq g_\infty)$. Let $f:\R\rightarrow \R$ be a bounded measurable function. Then under $\P^{\leftarrow a}$ the process $(X_t: t \leq g_\infty)$ is inhomogeneous and its transition probabilities are determined by
\begin{equation}\label{120}
      \E^{\leftarrow a}_{x}[f(X_{t})|\mathcal F_s; t<g_\infty]=\frac{\E_{X_{s}}[f(X_{t-s})V_{0,a-t-s}(X_{t-s})]}{V_{t,a-s}(X_s)} \qquad s \leq t \leq a.
 \end{equation}
    The latter process, $(X_t: t \geq g_\infty),$ has entrance law
    \[
      \E^{\leftarrow a}_x[f(X_{t+g_\infty})]=\n(f(\epsilon_t)h(\epsilon_t):t< \zeta),
    \]
    and semi-group given by Doob's h-transform
    $$\frac{1}{h(x)}\E_{x}\left(f(X_{t})h(X_{t}), t<T_{0}\right),\qquad \text{ for all }x\text{ with }h(x) \neq 0.$$
  \end{corollary}

  We shall spend the remainder of this section proving Theorem \ref{thm:markov_limit}, Theorem \ref{thm:markov_limit_desc} and Corollary~\ref{cor:1}. Our proof is similar to the proof of the conditioning we have seen in the previous part. We again use a technique similar to \cite{CD}, \cite{panti}, and \cite{rivero2005} .

   We will henceforth assume that the assumptions (A) and (B) hold. We begin with the following lemma.

  \begin{lemma}\label{lemma:h_q_excessive}
    We have that the function $h_q$, defined in \eqref{eq:h_q}, is excessive in the sense that
    \[
      \E_x[h_q(X_t); t < T_0] \leq h_q(x) \qquad t \geq 0.
    \]
  The same holds for the function $h$.
     Furthermore,
    \begin{equation}\label{16}
      \E_{x}\left(h_{q}(X_{t}); t<T_0\right)=e^{qt}h_{q}(x)-\frac{q}{q\beta+\n(\zeta>\e_q)}\int^t_0\P_x(T_0>u)e^{-qu}{\rm d}u.
    \end{equation}
  \end{lemma}
  The lemma follows essentially from the Markov property, see for example page 22 in \cite{panti} for further details.

  Next we decompose the process $(\P_x,X)$ into two processes; one process describes its law until time $g_{\e_q}$ and the other describes its law after time $g_{\e_q}$. The formula \eqref{eq:master_formula}  enables us to decompose the path of $(X_t: t \leq \e_q)$, conditionally on $g_{\e_q} < a$.

  \begin{lemma}\label{lemma:condition_q}
    Suppose that $F,G: \mathbb D \rightarrow \R$ are bounded measurable functionals and $f:\R\rightarrow \R$ is bounded and measurable. Then for every $r>0$,
    \begin{align}
        \E_x&[F(X_s: s \leq g_{\e_q})f(g_{\e_q})G(X_{s+g_{\e_q}}:s \leq r); r< \e_q-g_{\e_q}|g_{\e_q}<a]\nonumber\\
        &=\E_x\left[\int_0^a{\rm d}L_t e^{-qt}\frac{F(X_s:s \leq t)}{V^{(q)}_{0,a}(x)} f(t)\frac{\n(G(\epsilon_s:s \leq r)(e^{-qr}-e^{-q\zeta}); r<\zeta)}{q\beta + \n(1-e^{-q\zeta})}\right].\label{eq:markov_finite_q}
    \end{align}
  \end{lemma}
  \begin{proof}
    Fix $r>0$. For $v \geq 0$, consider the following functional
    \[
    F^{(v)}_u(\omega,\epsilon)= \1_{\{r+u<v<\zeta + u\}}\1_{\{u<a\}}f(u) F(\omega_s: s \leq u)G(\epsilon_s:s \leq r).
    \]
    Then we have that
    \begin{align*}
      \E_x&[F(X_s: s \leq g_{\e_q})f(g_{\e_q})G(X_{s+g_{\e_q}}:s \leq r); r< \e_q-g_{\e_q};g_{\e_q}<a]\\
      &=\int_0^\infty\E_x\left[\sum_{i=1}^\infty F^{(v)}_{\ell_i}((X_t:t \leq \ell_i),\epsilon_i)\right] qe^{-qv}\,{\rm d}v.
    \end{align*}
    It may be the case that $X_{\e_q}=0$ (when $\beta >0$), in which case $g_{\e_q}=\e_q$ and the above expectation is zero.
    Using \eqref{eq:master_formula} we get that
    \begin{align*}
      \E_x&[F(X_s: s \leq g_{\e_q})f(g_{\e_q})G(X_{s+g_{\e_q}}:s \leq r); r< \e_q-g_{\e_q};g_{\e_q}<a]\\
      &=\E_x \left[\int_{0}^{a}{\rm d}L_u F(X_s: s \leq u)f(u)\int_{\mathcal E^*}\n({\rm d}\epsilon; r <\zeta ) \int_{u+r}^{u+\zeta} {\rm d}v q e^{-qv}G(\epsilon_s:s \leq r)\right] \\
      &= \E_x \left[\int_{0}^{a}{\rm d}L_u e^{-qu}F(X_s: s \leq u)f(u)\int_{\mathcal E^*}\n({\rm d}\epsilon; r <\zeta ) e^{-qr}\int_{0}^{\zeta-r} {\rm d}v q e^{-qv}G(\epsilon_s:s \leq r)\right] \\
      &=\E_x \left[\int_{0}^{a}{\rm d}L_u e^{-qu}F(X_s: s \leq u)f(u)\right] \n(G(\epsilon_s:s \leq r)(e^{-qr}-e^{-q\zeta}); r<\zeta).
    \end{align*}
    Hence we are left to show that $\P(g_{\e_q} <a) = V^{(q)}_{0,a}(x)(q\beta + \n(1-e^{-q\zeta}))$.

    Taking $F=1$, $G=1$, $f=1$ and $r \downarrow 0$ in the equation above gives that
    \[
    \P_x(g_{\e_q}<a; X_{\e_q}\neq 0)= \E_x\left[\int_0^a {\rm d}L_u e^{-qu}\right]\n(1-e^{-q\zeta})=V^{(q)}_{0,a}(x)\n(1-e^{-q\zeta})
    \]
    where in the second equality we have used \eqref{eq:V_def}. Now it remains to show that $\P_x(g_{\e_q}<a; X_{\e_q}= 0)=V^{(q)}_{0,a}(x)q\beta$. Notice that $X_{\e_q}=0$ occurs if and only if $g_{\e_q}=\e_q$, hence we get that
    \begin{align*}
    \P_x(g_{\e_q}<a; X_{\e_q}= 0)&=\P_x(\e_q<a; X_{\e_q}=0)\\
    &=\int_0^a qe^{-qt} \P_x(X_t=0)\, {\rm d}t\\
    &=q\beta\E_x \left(\int_0^a e^{-qt}\, {\rm d}L_t \right)\\
    &=q\beta V^{(q)}_{0,a}(x)
    \end{align*}
    where in the second equality we have used \eqref{eq:local_def} and in the final equality we have again used \eqref{eq:V_def}. This concludes the proof.
  \end{proof}

Notice that Theorem \ref{thm:markov_limit} 
  immediately implies that if $\P^{\leftarrow a}_x$ exists, then $(X,\P^{\leftarrow a}_x)$ is a Markov process.

To prove the convergence in Theorem \ref{thm:markov_limit} notice that the factor on the left of \eqref{eq:markov_finite_q} converges to the desired limit given in Theorem \ref{thm:markov_limit_desc} as $q \downarrow 0$. The Markov property of the process $(X_{t}: t\leq g_{\infty})$ under $\P^{\leftarrow a}$ is easily deduced from equation \eqref{eq:markov_finite_q}. Indeed, we infer that for any $s<t<a$, and functionals $f:\mathbb{R}\to [0,\infty)$ and $G: \mathbb{D}[0,s] \rightarrow \R$ measurable and positive, we have the identity
\begin{equation*}
\begin{split}
\P^{\leftarrow a}_x\left(G(X_{u}, u\leq t), t<g_{\infty}\right)&=\E_x\left[\int_0^a{\rm d}L_v \frac{G(X_u: u \leq t)}{V_{0,a}(x)} 1_{\{t<v\}}\right]\\
&=\E_x\left[\frac{G(X_u: u \leq t)}{V_{0,a}(x)} \int_{t}^a{\rm d}L_v \right]\\
&=\E_x\left[ G(X_u: u \leq t)\frac{V_{t,a}(X_{t})}{V_{0,a}(x)}\right],
\end{split}
\end{equation*}
where in the first equality we applied the expression resulting from taking the limit as $q\to 0$ in \eqref{eq:markov_finite_q},   then the second identity follows from the fact that the only term influenced by the integral with respect to the local time is the indicator function, and finally the third equality follows from the Markov property at time $t.$ From the latter identity we infer that the law of $(X_{u}: u\leq g_{\infty})$ under $\P^{\leftarrow a}$ is that of the $h$-transform of $X,$ killed at time $a,$ using the space-time excessive function $(x,t)\mapsto V_{t,a}(x)$ and hence its semigroup is given by \eqref{120}.
  
Now it remains to describe the second term in the product in \eqref{eq:markov_finite_q}.
  \begin{lemma}\label{lemma:conv_q_ex}
    We have that for all $t \geq 0$ and $G: \mathbb{D}[0,t] \rightarrow \R$ continuous and bounded,
    \[
    \lim_{q\downarrow 0}\frac{\n(G(\epsilon_s:s \leq t)(e^{-qt}-e^{-q\zeta}); t<\zeta)}{q\beta + \n(1-e^{-q\zeta})} = \n(G(\epsilon_s:s \leq t)h(\epsilon_t); t<\zeta)
    \]
  \end{lemma}
  \begin{proof}
    Recall the definition of $h_q(x)$ in \eqref{eq:h_q}. Integrating out the exponential results in the following,
    \[
    h_q(x) = \frac{\E_x[1-e^{-qT_0}]}{q\beta + \n(1-e^{-q\zeta})}
    \]
    where, as before, $T_0=\inf\{s \geq 0: X_s=0\}$. Hence we have that using the Markov property,
    \begin{align*}
    \frac{\n(G(\epsilon_s:s \leq t)(e^{-qt}-e^{-q\zeta}); t<\zeta)}{q\beta + \n(1-e^{-q\zeta})}&=\frac{\n(G(\epsilon_s:s \leq t)e^{-qt}\E_{\epsilon_t}[1-e^{-qT_0}]; t<\zeta)}{q\beta + \n(1-e^{-q\zeta})}\\
    &=\n(G(\epsilon_s:s \leq t) h_q(\epsilon_t);t < \zeta).
    \end{align*}
    Assumptions (A) and (B) together now imply the lemma either through the dominated convergence theorem or the monotone convergence theorem.
  \end{proof}

  Now Theorem \ref{thm:markov_limit}, Theorem \ref{thm:markov_limit_desc} and Corollary~\ref{cor:1} follow from Lemma \ref{lemma:condition_q} and Lemma \ref{lemma:conv_q_ex} together with the Markov property of $\epsilon$ under $\n$, see e.g. \cite{blumenthal}.

  \begin{remark}
    We believe that the assumptions (A) and (B) are minimal conditions for Theorem \ref{thm:markov_limit} and Theorem \ref{thm:markov_limit_desc} to hold. These assumptions can be easily verified in the case when $X$ is transient because in that case $h(x)=\frac{\P_{x}(T_{0}=\infty)}{\n(\zeta=\infty)}$; when $X$ is a L\'evy process, see e.g. \cite{CD} and \cite{panti}, or when $X$ is a positive self-similar Markov process \cite{rivero2005}.  See the Remark~\ref{AB} below for further details.
  \end{remark}

  \begin{remark}\label{AB}
    To finish let us observe that verifying (A) and (B) is not necessarily a hard task.  Notice that we have the identity
      $$\frac{1}{q\beta+\n(\zeta>\ee_{q})}=\E\left(\int_{[0,\infty)}{\rm d}L_{s}e^{-qs}\right).$$
    If we denote ${\mathcal{V}}({\rm d}s)=\E({\rm d}L_{s}),$ we can then express the function $h_{q}$ as
    $$h_{q}(x)=q\int^{\infty}_{0}{\rm d}te^{-qt }\int^{t}_{0}{\mathcal{V}}({\rm d}s)\P_{x}(T_{r}>t-s),\qquad q>0, x\in E.$$ From this fact it can be seen that if the function $t\mapsto \int^{t}_{0}{\mathcal{V}}({\rm d}s)\P_{x}(T_{r}>t-s)$ is differentiable on $(0,\infty)$ then the function $h_{q}$ is non-increasing in $q.$ As with respect to (A), from this identity we can see that for instance if $0$ is positive recurrent, $\n(\zeta)<\infty,$ and $\E_{x}(T_{0})<\infty$ for all $x\in E,$ then the renewal theorem implies that
    $$
    \int^{t}_{0}{\mathcal{V}}({\rm d}s)\P_{x}(T_{r}>t-s)\xrightarrow[t\to\infty]{}\frac{1}{\n(\zeta)}\E_{x}(T_{r})=h(x),\qquad x\in E.
    $$

    Notice that identity \eqref{16} implies that in this case $h$ is strictly excessive when $0$ is positive recurrent.
    This makes sense since when the origin is positive recurrent, conditioning on avoiding the origin is costly and results in the process being killed in finite time.
  In the null-recurrent case, $\n(\zeta)=\infty$ and the condition (A) holds under the assumption that the tail distribution of $T_{0}$ is regularly varying. This time~\eqref{16} shows that the function $h$ is invariant and so the process conditioned to avoid zero has an infinite lifetime.
  \end{remark}

  \subsection{L\'evy processes}\label{victor}
    Using the previous methods we aim at building a L\'evy process $(X_{t}, t\geq 0)$ conditioned not to go above level $b$ and its maximum is achieved before time $a.$ We refer to Chapter VII in \cite{kyprianou} for background on fluctuation theory of L\'evy processes. In order to avoid some technicalities we will make the additional assumption that
    \[
    0 \text{ is regular for $(0,\infty)$ and $(-\infty,0)$.}
    \]
We will denote by $\widehat{X}$ the dual L\'evy process $\widehat{X}=-X,$ and by $\widehat{\P}$ its law, that is the push forward measure of the mapping $\widehat{X}$ under $\P.$ As usual, $\widehat{\P}_{x}$ denotes the law of the dual process started from $x$, that is the law of $x+\widehat{X}$ under $\widehat{\P}$.

    Let $S_{t}=\sup_{s\leq t}\{X_{s}\vee 0\}$, $t\geq 0$. The process $X$ reflected in its past supremum $(S_{t}-X_{t}: t\geq 0)$, is a strong Markov process with respect to the natural filtration $(\mathcal F_t: t \geq 0)$ generated by $X$.
    Similar to the previous section $(S_t-X_{t}: t\geq 0)$ admit a local time at $0$ and we denote this by  $\overline L = (\overline{L}_{t}, t\geq 0)$. The process $\overline L$ admits a right-inverse and we denote it by $\overline L^{-1}$. Finally we let $\overline{n}$ denote the excursion measure at $0$ for  $(S_t-X_{t}: t\geq 0)$. We will denote by $\underline{n}$ the excursion measure for the dual process $\widehat{X}$ reflected in its past supremum.

    Next let $V({\rm d}s,dx)$ denote renewal measure of the upward ladder process $(\overline{L}^{-1}, X_{\overline{L}^{-1}})$ given by
$$\iint_{[0,\infty)^2}V({\rm d}s,{\rm d}x)g(s,x)=\E\left(\int^{\infty}_{0}{\rm d}u1_{\{\overline{L}^{-1}_{u}\in {\rm d}s, X_{\overline{L}^{-1}_{u}}\in {\rm d}x\}}1_{\{\overline{L}^{-}_{u}<\infty\}}\right).$$ Using that $\overline{L}$ increases only at the times where $X$ reaches its supremum and making a change of variables, we infer that for any  $g:[0,\infty)^2\to[0,\infty)$ measurable, we have the identity
\begin{equation}\label{eq:22}
\iint_{[0,\infty)^2}V({\rm d}s,{\rm d}x)g(s,x)=\E\left(\int^{\infty}_{0}{\rm d}\overline{L}_s g(s,X_s)\right).
\end{equation}
 From Lemma~1 in \cite{chaumont-supremum} we obtain also  that
    \begin{equation}\label{chaumont:1}
      V({\rm d}s,{\rm d}x)=\underline{n}(\epsilon_{s}\in {\rm d}x, s<\zeta){\rm d}s.
    \end{equation}  For $0<a, b\leq \infty$, we define
    $$V_{q}([0,a)\times[0,b))=\int_{[0,a)\times[0,b)}e^{-qs}V({\rm d}s,{\rm d}x)=\int^{a}_{0}{\rm d}s e^{-qs}\underline{n}(\epsilon_{s}<b, s<\zeta),$$ and the
  upward renewal function $$V(x)=\E\left(\int^{\infty}_{0}{\rm d}\overline{L}_{s}1_{\{X_{s}\leq x\}}\right).$$
  Following \cite{CD}, we will denote by $\widehat{\P}^{\uparrow}_{x}$ the law of the dual process $\widehat{X}$ conditioned to stay positive, started from $x\geq 0.$ This measure satisfies that for any $t\geq 0,$
  \begin{equation*}
  {\rm d}\widehat{\P}^{\uparrow}_{x}|_{\mathcal{F}_{t}}=\begin{cases}
  \frac{V(X_{t})}{V(x)}\mathbf{1}_{(t<\tau^{-}_{0})}{\rm d}\widehat{\P}_{x}|_{\mathcal{F}_{t}}, &\text{if } x>0,\\
  V(X_{t})\mathbf{1}_{(t<\tau^{-}_{0})}{\rm d}\overline{n}|_{\mathcal{F}_{t}} &\text{if } x=0.
  \end{cases}
  \end{equation*}
For convenience,   we write $\widehat{\P}^{\uparrow}$ in place of $\widehat{\P}^{\uparrow}_{0}.$ Here we will denote
$$g_{t}=\sup\{s<t: S_{s}-X_{s}=0\},\qquad t\geq 0;$$ which is consistent with the notation in the previous section as it is the last visit to zero before time $t$ for the process reflected at the supremum.


  \begin{theorem}
      We have the following limit
      \begin{equation}\label{condtostaybelow}
        \begin{split}
          &\lim_{q\to 0}\E\left(F(X_{s}, 0\leq s< g_{\e_{q}})f(g_{\e_q}, S_{\e_q}) G(X_{g_{\e_q}}-X_{u+g_{\e_q}}, 0\leq u\leq T-g_{\e_{q}} ) | g_{\e_q}\leq a, S_{\e_q}\leq b\right)\\
          &=\frac{\E\left(\int^{a}_{0}{\rm d}\overline{L}_{t}F(X_{s}, 0\leq s< t)f(t, S_{t})\1_{\{S_{t}\leq b\}}\right)}{V([0,a]\times[0,b])}\times  \widehat{\E}^{\uparrow}\left(G(X_{u}, 0\leq u\leq T)\right),
        \end{split}
      \end{equation}
      for every $T>0,$ $F,G:\mathbb{D}\to \re$, $f:\re\to\re,$ bounded measurable functionals. The left factor of the above equation corresponds to the law of the L\'evy process killed at the last time where it hits its overall supremum, conditioned to have an overall supremum reached by time $a$ and whose value is below $b.$
    \end{theorem}
    \begin{proof}
      The following identity is obtained by now standard calculations using the compensation formula for the process $X$ reflected in its past supremum, see e.g. \cite{C} for similar computations,
      \begin{equation*}
        \begin{split}
          &\E\left(F(X_{s}, s< g_{\e_q})f(g_{\e_q},S_{\e_q})G(X_{g_{\e_q}}-X_{u+g_{\e_q}}, 0\leq u\leq \e_{q}-g_{\e_{q}})\right)\\
          &=\E\left(q\int^{\infty}_{0}{\rm d}te^{-qt}F(X_{s}, 0\leq s< t)G(\mathbf{0})f(t, S_{t})1_{\{X_{t}=S_{t}\}}\right)\\
          & +q\E\left(\sum_{t>0}1_{\{d_{t}>g_{t}\}}F(X_{s}, 0\leq s< g_{t})f(g_{t}, S_{g_{t}})\left(\int^{d_{t}}_{g_{t}}{\rm d}se^{-qs}G(X_{g_{t}}-X_{u+g_{t}}, 0\leq u\leq s-g_{t})\right)\right)\\
          &=\kappa(q,0)\E\left(\int^{\infty}_{0}{\rm d}\overline{L}_{t}e^{-q t}F(X_{s}, 0\leq s< t)f(t, S_{t})\right)\\
          &\qquad \times\left[\frac{q}{\kappa(q,0)} \left[\overline{a}G(\mathbf{0})+\overline{n}\left(\int^{\zeta}_{0}{\rm d}se^{-qs}G(\epsilon_{u}, 0\leq u\leq s )\right)\right]\right];
        \end{split}
      \end{equation*}
where $\mathbf{0}$ denotes the path that is equal to zero everywhere, and the coefficient $\overline{a}$ corresponds to the drift of the inverse local time at the supremum, which is zero because $X$ is assumed to be regular downwards. From this formula we deduce that for any $f$
      \begin{equation}\label{eq::18}
        \begin{split}
          &\E\left(F(X_{s}, 0\leq s< g_{\e_{q}})f(g_{\e_q}, S_{\e_q}) G(X_{g_{\e_q}}-X_{u+g_{\e_q}}, 0\leq u<\e_{q}-g_{\e_{q}}) | g_{\e_q}\leq a, S_{\e_q}\leq b\right)\\
          &=\frac{\E\left(\int^{a}_{0}{\rm d}\overline{L}_{t}e^{-q t}F(X_{s}, 0\leq s<t)f(t, S_{t})\1_{\{S_{t}\leq b\}}\right)}{V_{q}([0,a]\times[0,b])}\\
          &\qquad \times \left[\frac{q}{\kappa(q,0)}\overline{n}\left(\int^{\zeta}_{0}{\rm d}se^{-qs}G(\epsilon_{u}, 0\leq u<s)\right)\right]
        \end{split}
      \end{equation}
      We would like to determine the limit as $q\to 0$ of the above expressions. The monotone convergence theorem implies that the following limit holds
      \begin{equation*}
        \begin{split}
         &\lim_{q\to 0}\frac{\E\left(\int^{a}_{0}{\rm d}\overline{L}_{t}e^{-q t}F(X_{s}, 0\leq s<t )f(t,S_{t})\1_{\{S_{t}\leq b\}}\right)}{V_{q}([0,a]\times[0,b])}\\
          &=\frac{\E\left(\int^{a}_{0}{\rm d}\overline{L}_{t}F(X_{s}, 0\leq s<t)f(t,S_{t})\1_{\{S_{t}\leq b\}}\right)}{V([0,a]\times[0,b])},
        \end{split}
      \end{equation*}
      for any $F:\mathbb{D}\to\re$ and $f:\re\to\re$ positive and measurable functionals.

  Let us verify that as claimed, the measure under squared brackets in \eqref{eq::18} converges towards that of the dual L\'evy process $\widehat{X}$ conditioned to stay positive. For $T>0,$ we define a measure on $\mathcal{F}_{T}$ by setting
      $$\E^{\downarrow,T,q}(H(X_{s}, s\leq T)\1_{\{T<\zeta\}})=\frac{q}{\kappa(q,0)}\left[\overline{n}\left(\int^{\zeta}_{0}{\rm d}se^{-qs}H(\epsilon_{u}, u\leq T)\1_{\{T<s\}}\right)\right],$$ with $H:\mathbb{D}\to\re$ any positive measurable functional. Equivalently, $\E^{\downarrow,T,q}$ is the restriction to $\mathcal{F}_{T}\cap\{T<\zeta\},$ of the measure in the rightmost factor in \eqref{eq::18}. Also, by taking $F\equiv 1\equiv f$ in \eqref{eq::18}, we see that the latter is equal to the law of $(X_{g_{\e_{q}}}-X_{g_{\e_{q}+s}}, 0\leq s\leq T-g_{\e_{q}})$ conditionally on $\{g_{\e_{q}}\leq a, S_{\e_{q}}\leq b\}.$
  Recall that under $\overline{n}$ the canonical process of excursions has the strong Markov property with the same semigroup as the dual process $\widehat{X}$ killed at its first passage time below $0$; see for example Chapter VI.48 of \cite{RWII}. The Markov property at time $T$ implies hence that
      \begin{equation*}
        \begin{split}
          \frac{q}{\kappa(q,0)}\overline{n}\left(\int^{\zeta}_{0}{\rm d}se^{-qs}H(\epsilon_{u}, u\leq T)\1_{\{T<s\}}\right)&=\frac{q}{\kappa(q,0)}\overline{n}\left(\int^{\zeta}_{T}{\rm d}se^{-qs}H(\epsilon_{s}, s\leq T)\1_{\{T<s\}}\right)\\
          &=\overline{n}\left(H(\epsilon_{s}, s\leq T)\frac{\widehat{\P}_{\epsilon_{T}}(\tau^{-}_{0}>\e_q)}{\kappa(q,0)}\1_{\{T<\zeta\}}\right).
        \end{split}
      \end{equation*}
      The function $$h_{q}(x)=\frac{\widehat{\P}_{x}(\tau^{-}_{0}>\e_q)}{\kappa(q,0)},\qquad  x\geq 0,$$ is known to be an excessive function for the dual process killed at its first passage time below $0,$ and to be equal to
      $$h_{q}(x)=V_{q}((0,\infty)\times[0,x])=\E\left(\int^{\infty}_{0}{\rm d}\overline{L}_{s}e^{-qs}1_{\{X_{s}\leq x\}}\right),$$ see e.g. \cite{CD}. For each $x\geq 0,$ it converges monotonically increasing to $V(x)$ which is known to be invariant for the dual process $X$ killed at its first passage time below $0,$ unless the process drifts towards $-\infty,$ in which case the function is excessive. It follows that for every $H$ as above we have the convergence
      $$\E^{\downarrow,T}(H(X_{s}, s\leq T)\1_{\{T<\zeta\}}):=\lim_{q\to 0}\E^{\downarrow,T,q}(H(X_{s}, s\leq T)\1_{\{T<\zeta\}})=\widehat{\P}^{\uparrow}(H(X_{s}, s\leq T)).$$ The above relation defines a family of measures $(\E^{\downarrow,T}, T\geq 0)$ on $\mathcal{F},$ which is consistent. By the Kolmogorov consistency theorem the unique measure on $\mathcal{F}$ whose restriction to $\mathcal{F}_{T}$ is $\E^{\downarrow,T}$, for $T\geq 0$, coincides with $\widehat{\P}^{\uparrow}.$
       \end{proof}

  Our next aim is to describe in further detail the pre-supremum path of $X,$ $(X_{t}, t\leq g_{\infty}),$ under a probability measure $\P^{a,b}$ on $\mathcal{F},$ whose expectations are defined by
  \begin{equation*}
        \begin{split}
  \E^{a,b}\left(F(X_{s}, 0\leq s<g_{\infty})f(g_{\infty}, S_{g_{\infty}})\right)&:=\lim_{q\to 0}\frac{\E\left(\int^{a}_{0}{\rm d}\overline{L}_{t}e^{-q t}F(X_{s}, 0\leq s<t)f(t,S_{t})\1_{\{S_{t}\leq b\}}\right)}{V_{q}([0,a]\times[0,b])}\\
          &=\frac{\E\left(\int^{a}_{0}{\rm d}\overline{L}_{t}F(X_{s}, 0\leq s<t)f(t,S_{t})\1_{\{S_{t}\leq b\}}\right)}{V([0,a]\times[0,b])},
        \end{split}
      \end{equation*}
  with $F:\mathbb{D}\to\re, f:\re\to\re,$ positive measurable functionals, as above. The probability measure $\P^{a,b}$ is carried by the paths with lifetime bounded by $a$ and whose supremum does not exceed the level $b$. Furthermore, as a particular consequence of the definition of $\P^{a,b}$ and the identity in \eqref{chaumont:1} we deduce the following Corollary.
  \begin{corollary}\label{corollary:4.10} Under the measure $\P^{a,b}$ we have
  $$\P^{a,b}\left(g_{\infty}\in {\rm d}s, S_{g_{\infty}}\in {\rm d}y\right)=\frac{1}{V([0,a]\times[0,b])}1_{\{0<s\leq a, 0\leq y\leq b\}}{\rm d}s\underline{n}(\epsilon_{s}\in {\rm d}y, s<\zeta).$$
  \end{corollary}
  In the spirit of Theorem~\ref{pathdecomp} we now describe the law $\E^{a,b}$ conditionally on the event $\{g_{\infty}=t\},$ for $0<t<a.$
  \begin{theorem}\label{condaginf}
  Fix $b>0,$ $a>0,$ and $0<s<a.$ The function $h_{s}$ defined by $$h_{s}(t,x,y)=\underline{n}\left(x<\epsilon_{s-t}<b-y, s-t<\zeta\right),\qquad t<s, y<b, x\geq 0,$$ is such that
  $$\E(h_{s}(t, S_{t}-X_{t}, X_{t})\1_{\{S_{t}<b\}})=h_{s}(0,0,0)=\underline{n}(0<\epsilon_{s}<b, s<\zeta),\qquad \text{for } s>t.$$ The measure $Q^{s,b}$ defined on $\mathcal{F}_{s-}$ thorough the relation
  \begin{equation}\label{eq:19.1}
  Q^{s,b}(F(X_{u}, u\leq T)):=\E\left(F(X_{u}, u\leq T)\frac{h_{s}(T, S_{T}-X_{T}, X_{T})}{h_{s}(0,0,0)}\1_{\{S_{T}<b\}}\right),\qquad T<s,
  \end{equation} for any $F:\mathbb{D}\to\re^{+}$ measurable functional,
   is a regular conditional version of $\E^{a,b}$ given $\{g_{\infty}=s\}.$
  \end{theorem}
  \begin{proof}
  By the identity \eqref{condtostaybelow} and the Markov property for $X$ under $\P$
  we have
  \begin{equation*}
  \begin{split}
  &\E^{a,b}\left(F(X_{u}, u\leq T)1_{\{T<g_{\infty}\}}f(g_{\infty})\right)\\
  &=C\E\left(\int^{a}_{0}{\rm d}\overline{L}_{t}F(X_{u}, u\leq T)1_{\{T<t, S_{t}<b\}}f(t)\right)\\
  &=C\E\left(F(X_{u}, u\leq T)\1_{\{S_{T}<b\}}\E\left(\int^{a}_{0}{\rm d}\overline{L}_{t}1_{\{T<t, S_{t}<b\}}f(t)|\mathcal{F}_{T}\right)\right),
  \end{split}
  \end{equation*}
  where $C=\frac{1}{V([0,a]\times[0,b])}.$ To determine the conditional expectation we use the following common identity in fluctuation theory $$S_{T+u}-X_{T}=(S_{T}-X_{T})\vee\sup\{X_{T+v}-X_{T}, v\leq u\};$$ this together with the independence and stationarity of the increments and that the local time grows only at the instants where $X$ reaches a new supremum, allows to simplify this expression to get
  \begin{equation*}
  \begin{split}
  &\E\left(\int^{a}_{T}{\rm d}\overline{L}_{t}1_{\{S_{t}<b\}}f(t) | \mathcal{F}_{T}\right)\\
  &=\E\left(\int^{a-T}_{0}{\rm d}\overline{L}_{v}1_{\{x<S_{v}<b-y\}}f(v+T) \right)|_{\{x=S_{T}-X_{T}, y=X_{T}\}}.
  \end{split}
  \end{equation*}
  Using the equalities \eqref{eq:22} and \eqref{chaumont:1}, together with Fubini's Theorem the right most term above can be written as
  \begin{equation*}
  \begin{split}
  &\int^{a-T}_{0}{\rm d}v f(v+T)\underline{n}\left(x<\epsilon_{v}<b-y, v<\zeta\right)|_{\{x=S_{T}-X_{T}, y=X_{T}\}}\\
  &=\int^{a}_{T}{\rm d}s f(s)h_{s}(T, S_{T}-X_{T}, X_{T}).
  \end{split}
  \end{equation*}
  Putting the pieces together we infer
  \begin{equation*}\label{eq:20v1}
  \begin{split}
  &\E^{a,b}\left(F(X_{u}, u\leq T)1_{\{T<g_{\infty}\}}f(g_{\infty})\right)\\
  &=C\int^{a}_{0}{\rm d}s f(s)\underline{n}(0<\epsilon_{s}<b, s<\zeta)\E\left(F(X_{u}, u\leq T)\1_{\{S_{T}<b, T<s\}}\frac{h_{s}(T, S_{T}-X_{T}, X_{T})}{\underline{n}(0<\epsilon_{s}<b, s<\zeta)}\right).
  \end{split}
  \end{equation*}
Applying this formula for $T>0,$ $F\equiv1,$ and using Corollary~\ref{corollary:4.10} we deduce the identity
  $$\int^{a}_{T}{\rm d}s f(s)\underline{n}(0<\epsilon_{s}<b, s<\zeta)=\int^{a}_{T}{\rm d}s f(s)\E\left(\1_{\{S_{T}<b\}}h_{s}(T, S_{T}-X_{T}, X_{T})\right).$$ Since the above holds for any $f$ positive and measurable we deduce that for $T>0$ and a.e. $s>T$
  $$\underline{n}(0<\epsilon_{s}<b, s<\zeta)=\E\left(\1_{\{S_{T}<b\}}h_{s}(T, S_{T}-X_{T}, X_{T})\right).$$
  By the right continuity of $s\mapsto h_{s}(T,x,y)$ and the bound $$h_{s}(T,x,y)\leq \underline{n}(s-T<\zeta)\leq \underline{n}(\delta<\zeta)<\infty,\qquad \text{with $\delta>0$ s.t.} \  s-T>\delta,$$ it is seen using a dominated convergence argument that the latter identity holds for any $s>T.$ This implies the first claim in the Theorem. The second claim now follows from the identity \eqref{eq:20v1} and the Kolmogorov consistency theorem to ensure that there is a unique measure, $Q^{s,b},$ that satisfies the relation \eqref{eq:19.1}.
  \end{proof}
  It is possible to push forward the description of the measure $\E^{a,b}$ by conditioning on the value of the pair $(g_{\infty}, S_{g_{\infty}}).$ This needs for instance the further assumption that $X$ is such that its semigroup is absolutely continuous and with bounded densities, viz
  $$\P_{x}(X_{t}\in {\rm d}y)=p_{t}(y-x){\rm d}y,\qquad y\in \mathbb{R}, t\geq 0,$$ with $p_{t}(\cdot)$ bounded. In this setting it has been proved in \cite{chaumontmalecki} that the measures $\underline{n}(\epsilon_{s}\in {\rm d}y, s<\zeta)$ are absolutely continuous with respect to Lebesgue's measure
  $$\underline{n}(\epsilon_{s}\in {\rm d}y,s<\zeta)=q^{*}_{s}(y){\rm d}y,\qquad y>0, s>0,$$ and for $s>0,$ $q^{*}_{s}(\cdot)$ is a strictly positive and continuous function on $(0,\infty).$ Then Corollary~\ref{corollary:4.10} becomes
  $$\P^{a,b}\left(g_{\infty}\in {\rm d}s, S_{g_{\infty}}\in {\rm d}y\right)=\frac{1}{V([0,a]\times[0,b])}1_{\{0<s\leq a, 0\leq y\leq b\}}q^{*}_{s}(y){\rm d}s{\rm d}y.$$ Following the arguments in the proof of Theorem~\ref{condaginf} it is possible obtain a version of the formula \eqref{eq:20v1} but for $f(g_{\infty}, S_{g_{\infty}}),$ which will give place to an expression of the regular conditional version of $\E^{a,b}$ given $\{g_{\infty}=s, S_{g_{\infty}}=y\},$ with $s\leq a$ and $y\leq b.$ 

  \section*{Acknowledgements}

We would like to thank Ron Doney who suggested what conditioning a subordinator to stay in a strip would entail and the two anonymous referees for their valuable feedback.
AEK and B\c{S} acknowledge support from EPSRC grant number EP/L002442/1. AEK and VR acknowledge support from EPSRC grant number EP/M001784/1. VR acknowledges support from CONACyT grant number 234644. This work was undertaken whilst VR was on sabbatical at the University of Bath, he gratefully acknowledges the kind hospitality of the Department and University.


\printbibliography

\end{document}